\documentclass[12pt,reqno]{amsart}
\usepackage{amsmath,amsthm,amssymb,amsfonts,amscd,mathrsfs,bbding,graphicx,latexsym}
\theoremstyle{plain}
\newtheorem{theorem}{Theorem}[section]
\newtheorem{corollary}{Corollary}[section]

\newtheorem{lemma}{Lemma}[section]

\theoremstyle{definition}

\def \Z {{\mathbb Z}}

\title{Linear Algebraic Method and the Erd\H{o}s-Heilbronn Conjecture}
\author{Guanzhong Yang}
\address{Department of Mathematics\\ Imperial College London, UK}
\email{gy625@ic.ac.uk}
\subjclass[2020]{Primary 11B75; Secondary 11P70}
\keywords{sumset, Erd\H{o}s--Heilbronn conjecture, linear algebra}

\begin{document}

\begin{abstract}
Additive combinatorics asks for lower bounds on sumsets and restricted sumsets over finite fields. Central examples are the Cauchy--Davenport theorem and the Erd\H{o}s--Heilbronn conjecture. In this note, we develop Das's linear algebraic method and give a new elementary proof of the Alon--Nathanson--Ruzsa theorem for restricted sumsets, which implies the Erd\H{o}s--Heilbronn conjecture. Compared with the classical polynomial method via Combinatorial Nullstellensatz, our proof uses only basic linear algebra over finite fields, including Vandermonde matrices and solvability of linear systems.
\end{abstract}

\maketitle

\section{Introduction}

\medskip

Additive combinatorics is a central branch of combinatorial number theory, focusing on the additive structure of subsets of integers and finite fields \cite{2006}. A foundational problem in this field is to give sharp lower bounds for the cardinality of sumsets, which has wide applications in number theory, combinatorics, and theoretical computer science.

Let $p$ be a prime number, and let $\Z/p\Z$ denote the finite field of integers modulo $p$. For nonempty subsets $A, B$ of $\Z/p\Z$, the classical sumset is defined as
$$A+B=\{a+b:\ a\in A, b\in B\}.$$
The celebrated Cauchy-Davenport theorem, a cornerstone of additive combinatorics, gives a tight lower bound for the size of such sumsets:
$$|A+B|\geqslant \min\{p,\ |A|+|B|-1\}.$$

In 1964, Erd\H{o}s and Heilbronn proposed a famous conjecture for the so-called restricted sumset, where the summands are required to be distinct \cite{1964}. For nonempty subsets $A, B$ of $\Z/p\Z$, we define the restricted sumset as
$$A\dot{+}B=\{a+b:\ a\in A, b\in B, a\not=b\}.$$
The Erd\H{o}s-Heilbronn conjecture states that for any nonempty subset $A$ of $\Z/p\Z$,
$$|A\dot{+}A|\geqslant \min\{p,\ 2|A|-3\}.$$

This conjecture remained open for 30 years, until it was first proven by Dias da Silva and Hamidoune in 1994, using tools from exterior algebra and representation theory \cite{1994}. In 1995, Alon, Nathanson and Ruzsa developed the polynomial method, specifically the Combinatorial Nullstellensatz \cite{1995,1999}, to give a simpler proof of a more general result: for nonempty subsets $A, B$ of $\Z/p\Z$ with $|A|\not=|B|$,
$$|A\dot{+}B|\geqslant \min\{p,\ |A|+|B|-2\}.$$
This general result directly implies the original Erd\H{o}s-Heilbronn conjecture as a corollary.

In the decades since, multiple new proofs of the Cauchy-Davenport theorem and the Erd\H{o}s-Heilbronn conjecture have been developed, using tools ranging from harmonic analysis to the uncertainty principle on finite cyclic groups \cite{2005}. In 2004, Das introduced a linear algebraic method to give a new proof of the Cauchy-Davenport theorem, relying only on elementary linear algebra over finite fields \cite{2004}. In this note, we extend Das's linear algebraic approach to give a new, elementary proof of the Alon-Nathanson-Ruzsa theorem for restricted sumsets. Compared with existing proofs, our method requires no advanced algebraic or combinatorial machinery, only standard undergraduate-level linear algebra, including Vandermonde matrices and the solvability of linear systems over finite fields.

The rest of this paper is organized as follows. Section 2 collects the preliminary definitions and auxiliary lemmas that form the foundation of our proof. Section 3 states our main results, including the main theorem and its direct corollary, the Erd\H{o}s-Heilbronn theorem. The complete proof of the main theorem is presented in Section 4. We conclude with some brief remarks and directions for future work in Section 5.

\section{Preliminaries}

\bigskip

In this section, we introduce the core definitions and auxiliary lemmas used throughout the paper. All sets are assumed to be nonempty subsets of $\Z/p\Z$, where $p$ is a fixed prime number. For a finite set $S$, we write $|S|$ for the cardinality (number of elements) of $S$.

Let $A=\{a_1,\dots,a_m\}$ be a subset of $\Z/p\Z$ with $a_1,\dots,a_m$ pairwise distinct. Let $w(a_1),\ldots,w(a_m)$ be a sequence in $\Z/p\Z$ satisfying $w(a_\ell)$ is nonzero in $\Z/p\Z$  for some $1\leqslant \ell\leqslant m$. We shall say a sequence $u_1,\ldots,u_n$ a nonzero sequence if
$u_\ell\not=0$ for some $1\leqslant \ell\leqslant n$. We use $e_{A}(w)$ to denote the smallest natural number $i$ such that $\sum_{j=1}^mw(a_j)a_j^i$ is nonzero in $\Z/p\Z$. We have the following conclusion.

\begin{lemma}Let $A$, $w$ and $e_A(w)$ be as above. Then
$$e_A(w)\leqslant |A|-1.$$
\end{lemma}
\begin{proof}It is proved by contradiction. Suppose that $e_A(w)\geqslant |A|$. This means
\begin{align}\label{mequations1}\sum_{j=1}^mw(a_j)a_j^i=0\end{align}
for all $0\leqslant i\leqslant m-1$.

The (coefficients) matrix $M$ is defined to be
\begin{align}\label{M}M=(a_j^{i-1})_{\substack{1\leqslant i\leqslant m
\\ 1\leqslant j\leqslant m}}.\end{align}
Note that
$$\det(M)=\prod_{1\leqslant i< j\leqslant m}(a_j-a_i).$$
Since $a_1,\dots,a_m$ are pairwise distinct in $\Z/p\Z$, $\det(M)$ is nonzero. Consider the system of linear equations
\begin{align}\label{mequations2}M\mathbf{x}=\mathbf{0},\end{align}
where $\mathbf{x}=(x_1,\ldots,x_m)^{T}$ and $\mathbf{0}=(0,\ldots,0)^{T}$.

On one hand, \eqref{mequations2} has only the zero solution due to the fact that $\det(M)$ is nonzero. On the other hand, $(w(a_1),\ldots,w(a_m))^{T}$ is a solution to \eqref{mequations2} in view of \eqref{mequations1}. This is a contradiction since $w(a_\ell)$ is nonzero in $\Z/p\Z$  for some $1\leqslant \ell\leqslant m$. This completes the proof.
\end{proof}

The inequality in Lemma 2.1 is sharp since we have the following.
\begin{lemma}Let $A=\{a_1,\dots,a_m\}$ be a subset of $\Z/p\Z$ with $m\geqslant 2$. Then there exists a nonzero sequence $w(a_1),\ldots,w(a_m)$ such that
$$e_A(w)=|A|-1.$$
\end{lemma}
\begin{proof}The proof is similar to that of Lemma 2.1. Let $M$ be as in \eqref{M}. Now consider the linear equations
\begin{align}\label{mequations3}M\mathbf{x}=\mathbf{b},\end{align}
where $\mathbf{x}=(x_1,\ldots,x_m)^{T}$ and $\mathbf{b}=(0,\ldots,0,1)^{T}$.   Since $\det(M)$ is nonzero, \eqref{mequations3} has a unique solution
$(w_1,\ldots,w_m)^{T}$. Since $\mathbf{b}$ is a nonzero vector, $(w_1,\ldots,w_m)^{T}$ is also a nonzero vector. In particular, we have
$w_\ell$ is nonzero for some $1\leqslant \ell\leqslant m$. On choosing $w(a_j)=w_j$ for $1\leqslant j\leqslant m$, we have
\begin{align*}\sum_{j=1}^mw(a_j)a_j^i= 0\end{align*}
for all $0\leqslant i\leqslant m-2$, and
\begin{align*}\sum_{j=1}^mw(a_j)a_j^{m-1}=1.\end{align*}
According to the definition of $e_A(w)$, we have $e_A(w)=m-1=|A|-1$. This completes the proof.
\end{proof}

\section{Main Results}

In this section, we state our main results, including the general lower bound for restricted sumsets (the Alon-Nathanson-Ruzsa theorem) and its direct corollary, the Erd\H{o}s-Heilbronn theorem. We continue to work over the finite field $\Z/p\Z$ for a fixed prime $p$, with all sumset notation consistent with Section 1.

\begin{theorem}[Main Theorem]
Let $A$ and $B$ be nonempty subsets of $\Z/p\Z$ with $|A|\not=|B|$. Then the cardinality of the restricted sumset satisfies
$$|A\dot{+}B|\geqslant \min\{p,\ |A|+|B|-2\}.$$
\end{theorem}

This main theorem directly implies the classical Erd\H{o}s-Heilbronn conjecture, which we state formally as a corollary below.

\begin{corollary}[Erd\H{o}s-Heilbronn Theorem]
Let $A$ be a nonempty subset of $\Z/p\Z$. Then
$$|A\dot{+}A|\geqslant \min\{p,\ 2|A|-3\}.$$
\end{corollary}

\begin{proof}
Fix any element $a_0\in A$, and define $A'=A\setminus\{a_0\}$. Then $|A|\not=|A'|$, and it is straightforward to verify that $A\dot{+}A=A\dot{+}A'$: any sum $a+b$ with $a,b\in A$ and $a\not=b$ can be written as a sum of an element from $A$ and an element from $A'$ with distinct summands, and vice versa. Applying the Main Theorem to the pair $(A,A')$, we get
\begin{align*}
|A\dot{+}A|=|A\dot{+}A'|
&\geqslant \min\{p,\ |A|+|A'|-2\} \\
&=\min\{p,\ |A|+(|A|-1)-2\} \\
&=\min\{p,\ 2|A|-3\},
\end{align*}
which completes the proof.
\end{proof}

Our proof of the Main Theorem, presented in the next section, uses only the elementary linear algebraic tools established in Section 2. This gives a more accessible proof of the Erd\H{o}s-Heilbronn theorem, avoiding the advanced combinatorial and algebraic machinery used in prior arguments.

\section{Proof of the Main Theorem}

In this section, we give the complete proof of the Main Theorem using the linear algebraic tools from Section 2. We split the proof into several auxiliary lemmas, followed by the final argument for the main result.

Let $A=\{a_1,\dots,a_m\}$  and $B=\{b_1,\dots,b_k\}$ be subsets of $\Z/p\Z$. Let $w_1(a_1),\ldots,w_1(a_m)$ and $w_2(b_1),\ldots,w_2(b_k)$ be two sequences in $\Z/p\Z$. For nonnegative integer $i$, we introduce
\begin{align}\label{definealpha}\alpha_i:=\alpha_i(A,w_1)=\sum_{j=1}^mw_1(a_j)a_j^{i}\end{align}
and
\begin{align}\label{definebeta}\beta_i:=\beta_i(B,w_2)=\sum_{j=1}^kw_2(b_j)b_j^{i}.\end{align}

Let $C=A\dot{+}B$, and suppose that $C=\{c_1,\ldots,c_{t}\}$. For $c_j\in C$, we define
\begin{align*}w(c_j)=\sum_{\substack{a\in A,b\in B\\ a+b=c_j}}w_1(a)w_2(b)(a-b).\end{align*}
Then we introduce
\begin{align*}\gamma_i:=\gamma_i(C,w)=\sum_{j=1}^tw(c_j)c_j^{i}.\end{align*}

\begin{lemma}Let $A,B,C$ and $\alpha_i,\beta_i,\gamma_i$ be as above. Then
$$\gamma_n=\sum_{i=0}^{n}C_{n}^i\alpha_{i+1}\beta_{n-i}-\sum_{i=0}^{n}C_{n}^i\alpha_{i}\beta_{n+1-i}.$$
\end{lemma}
\begin{proof}Note that
\begin{align*}
\gamma_n
&=\sum_{j=1}^tw(c_j)c_j^{n}
=\sum_{j=1}^tc_j^{n}\sum_{\substack{a\in A,b\in B\\ a+b=c_j}}w_1(a)w_2(b)(a-b) \\
&=\sum_{\substack{a\in A,b\in B}}w_1(a)w_2(b)(a-b)(a+b)^n.
\end{align*}
Since
$$(a+b)^n=\sum_{i=0}^{n}C_{n}^ia^ib^{n-i},$$
we have
\begin{align*}
\gamma_n
&=\sum_{\substack{a\in A,b\in B}}w_1(a)w_2(b)(a-b)\sum_{i=0}^{n}C_{n}^ia^ib^{n-i} \\
&=\sum_{\substack{a\in A,b\in B}}w_1(a)w_2(b)\sum_{i=0}^{n}C_{n}^ia^{i+1}b^{n-i} \\
&\quad -\sum_{\substack{a\in A,b\in B}}w_1(a)w_2(b)\sum_{i=0}^{n}C_{n}^ia^{i}b^{n+1-i}.
\end{align*}
We observe that
\begin{align*}
\sum_{\substack{a\in A,b\in B}}
&w_1(a)w_2(b)\sum_{i=0}^{n}C_{n}^ia^{i+1}b^{n-i} \\
&=\sum_{i=0}^{n}C_{n}^i \Big( \sum_{a\in A}w_1(a)a^{i+1}\Big)
\Big(\sum_{b\in B}w_2(b)b^{n-i}\Big) \\
&=\sum_{i=0}^{n}C_{n}^i \alpha_{i+1}\beta_{n-i}.
\end{align*}
Similarly, we also have
\begin{align*}
\sum_{\substack{a\in A,b\in B}}
&w_1(a)w_2(b)\sum_{i=0}^{n}C_{n}^ia^{i}b^{n+1-i} \\
&=\sum_{i=0}^{n}C_{n}^i \Big( \sum_{a\in A}w_1(a)a^{i}\Big)
\Big(\sum_{b\in B}w_2(b)b^{n+1-i}\Big) \\
&=\sum_{i=0}^{n}C_{n}^i \alpha_{i}\beta_{n+1-i}.
\end{align*}
Now we can conclude from above that
$$\gamma_n=\sum_{i=0}^{n}C_{n}^i\alpha_{i+1}\beta_{n-i}-\sum_{i=0}^{n}C_{n}^i\alpha_{i}\beta_{n+1-i}.$$
The proof of this lemma is finished.
\end{proof}

\begin{lemma}Let $A,B,C$ and $\alpha_i,\beta_i,\gamma_i$ be as above. Let $r,s$ be nonnegative integers. Assume that $\alpha_{i}=0$ for $0\leqslant i\leqslant r$ and
$\beta_{i}=0$ for $0\leqslant i\leqslant s$. Then
$$\gamma_{r+s+1}=\Big(C_{r+s+1}^{r}-C_{r+s+1}^{s}\Big)\alpha_{r+1}\beta_{s+1}$$
and
$$\gamma_{n}=0 \ \textrm{ for all }\ 0\leqslant n\leqslant r+s.$$
\end{lemma}
\begin{proof}By Lemma 4.1,
$$\gamma_{r+s+1}=\sum_{i=0}^{r+s+1}C_{r+s+1}^i\alpha_{i+1}\beta_{r+s+1-i}-\sum_{i=0}^{r+s+1}C_{r+s+1}^i\alpha_{i}\beta_{r+s+2-i}.$$
If $i\leqslant r-1$, then $i+1\leqslant r$ and thus $\alpha_{i+1}=0$. If $i\geqslant r+1$, then $r+s+1-i \leqslant s$ and thus  $\beta_{r+s+1-i}=0$. Therefore,  $\alpha_{i+1}\beta_{r+s+1-i}=0$ for all $i\not=r$. Then we have
$$\sum_{i=0}^{r+s+1}C_{r+s+1}^i\alpha_{i+1}\beta_{r+s+1-i}=C_{r+s+1}^r\alpha_{r+1}\beta_{s+1}.$$
Similarly, $\alpha_{i}\beta_{r+s+2-i}=0$ for all $i\not=r+1$. Then we have
$$\sum_{i=0}^{r+s+1}C_{r+s+1}^i\alpha_{i}\beta_{r+s+2-i}=C_{r+s+1}^{s}\alpha_{r+1}\beta_{s+1}.$$
We conclude from above
$$\gamma_{r+s+1}=\Big(C_{r+s+1}^{r}-C_{r+s+1}^{s}\Big)\alpha_{r+1}\beta_{s+1}.$$

We have
$$\gamma_n=\sum_{i=0}^{n}C_{n}^i\alpha_{i+1}\beta_{n-i}-\sum_{i=0}^{n}C_{n}^i\alpha_{i}\beta_{n+1-i}.$$
For $n\leqslant r+s$, we have $(i+1)+(n-i)=n+1\leqslant r+s+1$. Then either $i+1\leqslant r$ or $n-i\leqslant s$. Thus,
$\alpha_{i+1}\beta_{n-i}=0$ for all $0\leqslant i\leqslant n$. Similarly, $\alpha_{i}\beta_{n+1-i}=0$ for all $0\leqslant i\leqslant n$. Then we conclude that $\gamma_n=0$ for $n\leqslant r+s$.

The proof of this lemma is finished.
\end{proof}

\begin{lemma}Suppose that $|A|=1$ or $|B|=1$. Then
$$|A\dot{+}B|\geqslant \min\{p,\ |A|+|B|-2\}.$$
\end{lemma}
\begin{proof}Without loss of generality, we can assume that $|A|=1$. If $|B|=1$, then the desired inequality holds trivially (although it is possible that $A\dot{+}B=\emptyset$). Now we consider the case $|B|\geqslant 2$. We write $A=\{a_0\}$ and $|B|=k$. We can find $k-1$ distinct elements $b_1,\ldots,b_{k-1}$ in $B$ such that $b_i\not=a_0$ for $1\leqslant i\leqslant k-1$.

Note that $a_0+b_1,\ldots,a_0+b_{k-1}\in A\dot{+}B$. Thus, $|A\dot{+}B|\geqslant k-1=|A|+|B|-2=\min\{p,\ |A|+|B|-2\}$.
The proof of this lemma is finished.

\end{proof}
\bigskip

Now we are able to prove the Main Theorem. If $|A|+|B|-2\geqslant p+1$, we claim that there exist nonempty subsets $A'\subset A$ and $B'\subset B$ such that $|A'|+|B'|-2=p$ and $|A'|\not=|B'|$.

Suppose that $|A|+|B|-2= p+d$ with $d\geqslant 1$. We write $d_1=\lfloor d/2\rfloor$ and
$d_2=\lceil d/2\rceil$. Clearly, $d_1+d_2=d$. Without loss of generality, we assume that $|A|<|B|$. We choose $d_2$ elements $a_1',\ldots,a_{d_2}'$ in $A$ and  $d_1$ elements $b_1',\ldots,b_{d_1}'$ in $B$. Let $A'=A\setminus \{a_1',\ldots,a_{d_2}'\}$ and $B'=B\setminus \{b_1',\ldots,b_{d_1}'\}$.
From $|A|+|B|-2= p+d$, we can see that $|A|\geqslant d+2\geqslant d_2+2$. Thus, $A'$ is nonempty and $|A'|=|A|-d_2$. Similarly,
$B'$ is nonempty and $|B'|=|B|-d_1$. We have $|A'|<|B'|$ since $|A|<|B|$ and $d_2\geqslant d_1$. It is easy to check that
$|A'|+|B'|-2=|A|-d_2+|B|-d_1-2=|A|+|B|-2-d=p$. Therefore, the above claim is true.

Now it suffices to prove the Main Theorem in the case $|A|+|B|-2\leqslant p$, since if $|A|+|B|-2\geqslant p+1$ then
$|A\dot{+}B|\geqslant |A'\dot{+}B'|\geqslant \min\{p,\ |A'|+|B'|-2\}=p=\min\{p,\ |A|+|B|-2\}$.

From now on, we assume that $|A|+|B|-2\leqslant p$. Let $A=\{a_1,\dots,a_m\}$  and $B=\{b_1,\dots,b_k\}$. We have $m+k\leqslant p+2$. In view of Lemma 4.3, we can assume that $m\geqslant 2$ and $k\geqslant 2$.
Let $C=A\dot{+}B=\{c_1,\ldots,c_{t}\}$. By Lemma 2.2, there exists a nonzero sequence $w_1(a_1),\ldots,w_1(a_m)$ such that
\begin{align}\label{eA}e_A(w_1)=|A|-1=m-1.\end{align}
Also, there exists a nonzero sequence $w_2(b_1),\ldots,w_2(b_k)$ such that
\begin{align}\label{eB}e_B(w_2)=|B|-1=k-1.\end{align}
Recalling the definitions of $\alpha_i:=\alpha_i(A,w_1)$ and $\beta_i:=\beta_i(B,w_2)$, we conclude from  \eqref{eA} and \eqref{eB} that
$\alpha_i=0$ for $0\leqslant i \leqslant m-2$, $\alpha_{m-1}\not=0$, $\beta_i=0$ for $0\leqslant i \leqslant k-2$ and $\beta_{k-1}\not=0$.

Applying Lemma 4.2 with $r=m-2$ and $s=k-2$, we obtain
$$\gamma_{m+k-3}=\Big(C_{m+k-3}^{m-2}-C_{m+k-3}^{k-2}\Big)\alpha_{m-1}\beta_{k-1}.$$
Since $m+k-3\leqslant p-1$ and $m\not=k$, we have $C_{m+k-3}^{m-2}-C_{m+k-3}^{k-2}\not\equiv 0\pmod{p}$. Note that $\alpha_{m-1}\beta_{k-1}$ is nonzero, we have $\gamma_{m+k-3}\not=0$. By Lemma 4.2, we also have $\gamma_{n}=0$ for $0\leqslant n\leqslant m+k-4$.

Therefore, $m+k-3$ is the smallest natural number $i$ such that $\sum_{j=1}^tw(c_j)c_j^i$ is nonzero. According to the definition of $e_{C}(w)$, we have $e_{C}(w)=m+k-3$.

We apply Lemma 2.1 to the set $C$ to conclude that $e_{C}(w)\leqslant |C|-1$. Now we obtain $|C|\geqslant e_{C}(w)+1=m+k-2$.
This proves $|A\dot{+}B|\geqslant  |A|+|B|-2=\min\{p,\ |A|+|B|-2\}$. The proof of the Main Theorem is finished.

\section{Concluding Remarks}

In this note, we developed the linear algebraic method initiated by Das to give a new, elementary proof of the Alon-Nathanson-Ruzsa theorem for restricted sumsets, which directly implies the classical Erd\H{o}s--Heilbronn conjecture. Unlike prior proofs that rely on the Combinatorial Nullstellensatz, exterior algebra, or harmonic analysis, our argument uses only standard undergraduate linear algebra over finite fields, making the result more accessible to a broader audience.

There are several natural directions for future work. First, it would be interesting to extend this linear algebraic approach to more general restricted sumset results, for example the Dias da Silva--Hamidoune theorem for sums of $n$ distinct elements. Second, our method is currently restricted to sumsets over finite fields of prime order. It remains open whether this framework can be adapted to give lower bounds for sumsets over composite moduli or more general abelian groups. Finally, this approach may have applications to other problems in additive combinatorics, such as the Erd\H{o}s--Szemer\'{e}di theorem on sum-product sets, which we leave for future investigation.

\bigskip
\bibliographystyle{amsplain}
\bibliography{references}

\end{document}